\documentclass{amsart}
\usepackage{color}
\usepackage{stmaryrd}
\usepackage{bbm}
\usepackage[capitalise]{cleveref}
\usepackage{amssymb,latexsym,amsmath, amsfonts}

\newcommand{\mL}{\mathcal L}

\newcommand{\mM}{\mathcal M}
\newcommand{\Lw}{\mathcal{L}_{g}}
\newcommand{\lw}{{\ell}_{g}}
\newcommand{\nw}[1]{\| {#1} \|_{\Lw}}
\newcommand{\Ra}{{\rm Range}(I-S_+)}

\newcommand{\N}{{\mathbb N}}
\newcommand{\Z}{{\mathbb Z}}
\newcommand{\R}{{\mathbb R}}
\newcommand{\C}{{\mathbb C}}
\newcommand{\cM}{{\mathcal M}}

\newcommand{\cL}{{\mathcal L}}
\newcommand{\cD}{{\mathcal D}}

\newcommand{\eli}{\ell_\infty(\mathbb{Z})}

\newcommand{\smt}{ S(\cM,\tau)}

\DeclareMathOperator{\tr}{tr}

\newtheorem{theorem}{Theorem}[section]
\newtheorem{lemma}[theorem]{Lemma}
\newtheorem{corollary}[theorem]{Corollary}
\newtheorem{remark}[theorem]{Remark}
\newtheorem{proposition}[theorem]{Proposition}
\newtheorem{definition}[theorem]{Definition}
\newtheorem{example}[theorem]{Example}

\title{Symmetric functionals on simply generated symmetric spaces}
\author{Galina Levitina}
\address{Galina Levitina\\ Mathematical Sciences Institute,
	The Australian National University,
	Canberra ACT 2601, Australia.}
\email{galina.levitina@anu.edu.au}

\author{Alexandr Usachev}
\address{ Alexandr Usachev\\ Central South University, Hunan, China.}
\email{dr.alex.usachev@gmail.com}
\date{\today}

\begin{document}

\begin{abstract}
In the present paper we suggest a construction of symmetric functionals on a large class of symmetric spaces over a semifinite von Neumann algebra. This approach establishes a bijection between the symmetric functionals on symmetric spaces and shift-invariant functionals on the space of bounded sequences. 
It allows to obtain a bijection between the classes of all continuous symmetric functionals on different symmetric spaces. Notably, we show that this mapping is not bijective on the class of all Dixmier traces. As an application of our results we prove an extension of the Connes trace formula for wide classes of operators and symmetric functionals.

Keywords: singular trace, symmetric functional, Banach limit, Connes trace formula.

Mathematics  Subject  Classifications  (2010): 47B10, 47G30,\\ 58B34, 58J42
\end{abstract}

\maketitle

\section{Introduction}
Let $B(H)$ denote the algebra of all bounded linear operators on a separable Hilbert space $H$. By $\{\mu(n, X)\}_{n\ge0}$ we denote the sequence of singular values of a compact operator $X \in B(H)$. 
Let $\psi: [0, \infty)\to (0,\infty)$ be a concave function such that $\psi(0)=0$ and $\lim_{t\to\infty} \psi(t) =\infty$. For a suitable (translation and dilation invariant, to be precise) bounded linear functional $\omega$ on the space of bounded sequences, J. Dixmier~\cite{D} proved that the weight 
\begin{equation}\label{Dix_original_construction}
	{\rm Tr}_\omega (X):= \omega \left( n \mapsto \frac{1}{\psi(1+n)} \sum_{k=0}^n\mu(k,X)\right), 
\end{equation}
defined on the positive cone of the Lorentz space
$$
\mathcal M_{\psi}:= \left\{X \in B(H) \text{ is compact}:  \sup_{n \ge 0}  \frac{1}{\psi(1+n)} \sum_{k=0}^n\mu(k,X) < \infty\right\}
$$
extends by linearity to a trace (symmetric functional, in other words) of the whole space $\mathcal M_{\psi}$. This was the first example of singular (that is, non-normal) traces.
The theory experienced its rebirth some thirty
years ago when A. Connes made singular traces a cornerstone of his 
noncommutative geometry and calculus \cite{C_AF}.
Since then, singular traces found their applications in mathematics as well as in theoretical physics (see e.g. \cite{ C_book, LSZ, PSTZ, SU2}). The development of noncommutative geometry instigated the formulation of its semifinite counterpart, where traces are replaced with symmetric functionals. This naturally led to the construction of Dixmier traces on Lorentz spaces over a semifinite von Neumann algebra. Since late 2000s it became clear that Lorentz spaces admit traces, distinct from that of Dixmier.

%
The first systematic construction of all traces on all operator ideals in Hilbert (and even Banach) spaces was developed by A. Pietsch in a series of papers \cite{P_trI, P_trII, P_trIII}. Precisely, he proved a bijection between ideals in $B(H)$ and so-called shift-monotone ideals in the space $c_0(\Z_+)$ of all bounded sequences vanishing at infinity as well as between traces on operator ideals and shift-invariant functionals on the corresponding shift-monotone ideals. In a recent paper this construction was extended to symmetric functionals on symmetric spaces over a semifinite von Neumann algebra \cite{LU}, where the corresponding shift-invariant functionals act on shift-monotone spaces of sequences indexed by $\Z$.

In a particular case of the weak trace class ideal 
$$\mathcal L_{1,\infty}:= \left\{A \in B(H) \text{ is compact}: \sup_{n \ge 0}  (1+n) \mu(n,X) < \infty\right\},
$$
a modification of the Pietsch approach, proved in \cite{SSUZ},  established a correspondence of all symmetric functionals on $\mathcal L_{1,\infty}$ and all shift-invariant functionals on bounded sequences.
In particular, positive normalised traces correspond to well-known Banach limits (see e.g. \cite{SSU_Notices}). 
A further significance of this approach was evidences by applications to integral operators in~\cite{SSUZ} (see also \cite{P_PDO}).

In the current paper we present a modification of the above approaches. Let  $\mathcal{M}$ be
a semifinite von Neumann algebra on $H$, equipped with a faithful normal semifinite trace $\tau$. For a decreasing function $g : (0,\infty) \to (0,\infty)$ such that 
\begin{equation*}
	\sup_{t>0} \frac{g(t)}{g(2t)} \le C.
\end{equation*}
we consider the following ``simply generated'' symmetric spaces
\begin{equation*}
	\Lw(\mathcal M, \tau) := \{ X \in \smt : \nw{A} := \sup_{t>0} \frac{\mu(t,X)}{g(t)} < \infty \},
\end{equation*}
where $t\mapsto \mu(t,X)$ is the generalised singular values function (see Preliminaries).
In the case when $(\mathcal M, \tau) = (B(H), \rm Tr)$, symmetric functionals on these spaces were studied in \cite{Pietsch_simply_gen, Pietsch_simply_gen_more}.
Note that the condition~\eqref{g} guarantees that $\Lw(\mathcal M, \tau)$ is quasi-Banach (see e.g.~\cite[Theorem 1.2]{Sparr}). The condition \eqref{g} is the minimal one, which we suppose throughout the paper.

%

Using a simple lemma (\cref{the map of anger}) we establish (see \cref{Bijective correspondence}) a bijection between all \textit{continuous} symmetric functionals on a simply generated symmetric space and all \textit{continuous} shift-invariant functionals on $\ell_\infty(\Z)$. 

Although this modified approach constructs continuous symmetric functionals on simply generated spaces only, it allows us to show that there exists a bijection between continuous symmetric functionals on any two arbitrary (subject to some mild conditions) simply generated spaces $\Lw$ and $\mathcal L_f$ (see \cref{bij}). 
Notably, in general, this bijective mapping is not a bijection between classes of Dixmier traces on distinct simply generated spaces $\Lw$ and $\mathcal L_f$ (see \cref{Dix_cor}).


In the last section we apply our results to a study average spectral asymptotics of integral operators. 
It was proved by A. Connes~\cite{C_AF} that the Dixmier trace of a (compactly supported) classical pseudo-differential operator of order $-d$ on $\R^d$ is proportional to its Wodzicki's residue.
A reason for generalisations and extensions of Connes' trace theorem is to enable computation of (Dixmier and/or other) symmetric functionals
of pseudo-differential operators that are not classical or, even more generally of operators which are not pseudo-differential.
Since the Connes seminal paper there were a number of attempts to extend Connes' trace theorem to classes wider than that of classical PDOs. 
For instance, it was established for the class of anisotropic PDOs in \cite{BN}, the perturbed Laplacian on the noncommutative two tori in \cite{Fathizadeh} and the class of Laplacian modulated operators on $\R^d$ in~\cite{KLPS} (see also~\cite{LPS} and~\cite[Section 11]{LSZ}). The later result was further extended in \cite{GU}.

We prove the extension of the Connes trace formula for the wide class of integral operators. This result extends~\cite[Theorem 8.10]{SSUZ} to the case of operators from $\Lw$ and~\cite[Theorem 9.1]{GU} to the case of wider class of symmetric functionals (all continuous and normalised). Furthermore, our result relates positive normalised symmetric functionals to Banach limits. This allows us to describe the measurability of compact operators in terms the classical almost convergence introduces by G.G. Lorentz in 1948 \cite{L}. 
%
\section{PRELIMINARIES}
In the present section we collect some definitions and results, which will be used throughout the paper.

For details on von Neumann algebra
theory, the reader is referred to e.g. \cite{Dix, KR1, KR2}
or \cite{Ta1}. For the complete exposition of the theory of noncommutative integration we refer to \cite{DPS}.
For convenience of the reader, some of the basic definitions are recalled here.


\subsection{Operator spaces}

In what follows,  $H$ is a complex separable Hilbert space and $B(H)$ is the
$*$-algebra of all bounded linear operators on $H$, equipped with the uniform norm $\|\cdot\|_\infty$, and
$\mathbf{1}$ is the identity operator on $H$. Throughout the paper we assume that  $\mathcal{M}$ is
a semifinite von Neumann algebra on $H$, equipped with a faithful normal semifinite trace $\tau$. In the case when $\tau$ is  a finite trace, we assume that it is a state, that is $\tau(\mathbf1)=1$. In addition, we assume that the von Neumann algebra $\cM$ is atomless or atomic algebra with atoms of equal trace.

An operator $x$ in $H$ with domain $\mathfrak{D}\left(X\right)$ is said to be affiliated with $\mathcal{M}$
if $YX\subseteq XY$ for all $Y\in \mathcal{M}^{\prime }$, where $\mathcal{M}^{\prime }$ is the commutant  of $\mathcal{M}$, defined as $\mathcal{M}^\prime=\{T\in B(H): TS=ST\quad  \forall S\in \mathcal{M}\}$.
An operator $X$  affiliated with $\cM$ with domain $\mathfrak{D}\left(X\right)$ is called  $\tau$-measurable if there exists a sequence
$\left\{P_n\right\}_{n=1}^{\infty}$ of projections such that
$P_n\uparrow \mathbf{1},$ $P_n\left(H\right)\subseteq \mathfrak{D}\left(X\right)$ and
$\tau(\mathbf{1}-P_n)<\infty $ for all $n.$ Equivalently, an operator $X$ affiliated with $\cM$, is $\tau$-measurable, if there exists $a>0$, such that $\tau(E^{|X|}(a,\infty))<\infty$, where $E^{Y}$ denotes the spectral measure of a self-adjoint operator $Y$ (see e.g. \cite[Proposition 2.3.6]{DPS}). 
The collection of all $\tau $-measurable is denoted by $S\left( \mathcal{M}, \tau\right)
$. It is a unital $\ast $-algebra
with respect to the strong sums and products (denoted simply by $X+Y$ and $XY$ for all $X,Y\in S\left( \mathcal{M}\right) $) \cite[Corollary 5.2]{Se}.

The generalized singular value function $\mu(X):t\rightarrow \mu(t,X)$, $t>0$,  of an operator $X\in \smt$
is defined by setting
\begin{equation}\label{def_mu}
	\mu(t,X)=\inf\{\|XP\|:\ P=P^*\in\mathcal{M}\mbox{ is a projection,}\ \tau(\mathbf{1}-P)\leq t\}.
\end{equation}

The generalised singular value function can be equivalently defined via the distribution function. For every self-adjoint
operator $X\in S(\mathcal{M},\tau),$ the distribution function is defined by  setting
$$d_X(t)=\tau(E^{X}(t,\infty)),\quad t>0.$$
In this case (see \cite[Proposition 2.2]{FigielKalton}), the generalised singular values function is defined as right-continuous inverse of the distribution function, namely
$$\mu(t,X)=\inf\{s\geq0:\ d_{|X|}(s)\leq t\}.$$

If $\mathcal{M}=B(H)$ and $\tau$ is the
standard trace $\tr$, then it is not difficult to see that
$S(\mathcal{M},\tau)=\mathcal{M}$ and for $X\in S(\mathcal{M},\tau)$ we have
$$\mu(t,X)=\mu(n,X),\quad t\in[n,n+1),\quad  n\geq0.$$
The sequence $\{\mu(n,X)\}_{n\geq0}$ is just the sequence of singular values of the operator $X.$

In the special case, when $\mathcal{M}=L_\infty(0,\infty)$ is the von Neumann algebra of all
Lebesgue measurable essentially bounded functions on $(0,\infty)$ acting via multiplication on the Hilbert space
$H=L_2(0,\infty)$, with the trace given by integration
with respect to Lebesgue measure $m$, the algebra $S(\cM,\tau)$ can be identified with the algebra
\begin{equation}\label{S}
	S(0,\infty)=\{f \text{ is measurable:} \ \exists A\in\Sigma, m((0,\infty)\setminus A)<\infty, f\chi_A\in L_\infty(0,\infty)\},
\end{equation}
where $\chi_{A}$ denotes the characteristic function of a set $A\subset(0,\infty)$.
The singular value function $\mu(f)$ defined above is precisely the
decreasing rearrangement $f^*$ of the function $f\in S(0,\infty)$ given by
$$f^*(t)=\inf\{s\geq0:\ m(\{|x|\geq s\})\leq t\}.$$

We recall that $S_0(\cM,\tau)$ denotes the space of all $\tau$-compact operators
 defined as 
 $$S_0(\cM,\tau) = \{X\in S(\cM,\tau) :  \ \mu(\infty, X)=\lim_{t\to\infty} \mu(t,X)=0\}.$$
Note that $X\in S_0(\cM,\tau)$ if and only if $\tau(E^{|X|}(a,\infty))<\infty$ for any $a>0$ (see e.g. \cite[Section 2.4]{DPS}).

We now recall the notion of symmetrically quasi-normed spaces.

\begin{definition}\label{symmetric_dfn}
	A quasi-normed space $(E(\cM,\tau),\|\cdot\|_{E})$ of $\smt$ is called a \emph{symmetric quasi-normed space} if  $Y\in E(\cM,\tau)$, $X\in\smt$ and $\mu(X)\leq \mu(Y)$ implies that $X\in E(\cM,\tau)$ and $\|X\|_{E}\leq \|Y\|_E$.  In the special case when $\cM=L_\infty(0,\infty)$ or $\cM=L_\infty(0,1)$ (respectively, $\cM=\ell_\infty(\Z_+)$) we use the term symmetric quasi-normed function (respectively, sequence) space instead and denote $E(\cM,\tau)$ by $E(0,\infty)$ (respectively, $E(0,1),$ or $E(\Z_+)$). 
\end{definition} 

In the special case, when $\cM=B(H)$ equipped with the standard trace $\tr$, we shall use the notation $E(H)$ instead of $E(B(H),\tr)$.

Note that any symmetric quasi-normed space $E(\cM,\tau)$ is necessarily a quasi-normed $\cM$-bimodule, that is, if $X\in E(\cM,\tau)$, $A,B\in \cM$, then $AXB\in E(\cM,\tau)$ and $\|AXB\|_{E}\leq \|A\|_\infty \|B\|_\infty \|X\|_{E}$. 

An example of symmetric quasi-normed  spaces are the noncommutative $\cL_p$-spaces, defined as 
$$\cL_p(\cM,\tau):=\{X\in \smt: \mu(X)\in L_p(0,\infty)\},\quad \|X\|_{L_p}=\tau(|X|^p)^{1/p},$$
where $0<p<\infty$ and $L_p(0,\infty)$ are the classical Lebesgue $L_p$-spaces.
%

The symmetric quasi-normed spaces we consider in the present paper are introduced in the following definition. 
\begin{definition}\label{def_L_g}
Let 
$g : (0,\infty) \to (0,\infty)$ be a decreasing vanishing at $+\infty$ function such that 
\begin{equation}\label{g}
	\sup_{t>0} \frac{g(t)}{g(2t)} \le C.
\end{equation}
We set 
\begin{equation}\label{L_weak}
	\Lw(\mathcal M, \tau) := \{ X \in \smt : \nw{A} := \sup_{t>0} \frac{\mu(t,X)}{g(t)} < \infty \}.
\end{equation}
\end{definition}
Everywhere throughout the paper we assume that $g$ is a function satisfying the assumption of \cref{def_L_g}.

It is clear that $(\Lw(\mathcal M, \tau),\nw{\cdot})$ is a symmetric quasi-normed space. Furthermore, since $\lim_{t\to\infty} g(t)=0$, it follows that $\mu(t,X)\to 0$ as $t\to\infty$ for any $X\in \Lw(\mathcal M, \tau)$. In particular, it follows that $\Lw(\mathcal M, \tau)\subset S_0(\cM,\tau)$.

%

We now prove a simple auxiliary lemma.
\begin{lemma}\label{lem_op_decomp}
Suppose that $\cM$ is atomless or atomic algebra with atoms of equal trace and let $E(\cM,\tau)$ be a symmetric quasi-normed space. 
\begin{enumerate}
\item Any $X\in E(\cM,\tau)$ can be written as $X=X_1+X_2$ with $X_1\in E(\cM,\tau)\cap \cL_1(\cM,\tau)$ and $X_2\in E(\cM,\tau)\cap \cM$. 
Moreover, $X_1$ and $X_2$ can be chosen as $X_1=XE^{|X|}(a,\infty), X_2=XE^{|X|}[0,a],$ where $a$ is such that $\tau(E^{|X|}(a,\infty))<\infty$;
\item Suppose, in addition,  that $E(\cM,\tau)\subset S_0(\cM,\tau)$. Then for any  $0\le X,Y\in E(\cM,\tau)$ such that $\mu(X)=\mu(Y),$ we have that 
$$\mu(XE^{|X|}[0,a])=\mu(YE^{|Y|}[0,a]),\quad \mu(XE^{|X|}(a,\infty))=\mu(YE^{|Y|}(a,\infty))$$
for any $a>0$.
\end{enumerate}
\end{lemma}
\begin{proof}
(i). 
Let $X\in E(\cM,\tau)$ be arbitrary. Since $X$ is necessarily $\tau$-measurable, it follows that there exists $a>0$, such that $\tau(E^{|X|}(a,\infty))<\infty$.  We set 
$$X_1=XE^{|X|}(a,\infty),\quad  X_2=XE^{|X|}[0,a],$$
and claim that $X=X_1+X_2$ is the desired decomposition. 
Since $E(\cM,\tau)$ is an $\cM$-bimodule, it follows that $X_1,X_2\in E(\cM,\tau)$. By the spectral theorem, we have that $X_2\in \cM$. Thus, it remains to show that $X_1\in \cL_1(\cM,\tau)$. 

Since $\cM$ is an atomless or atomic algebra with atoms of equal trace, it follows that $E(\cM,\tau)\subset \cL_1(\cM,\tau)+\cM$ (see e.g. \cite[Section 4.4]{DPS}).  Therefore, it follows from  \cite[Theorem 3.9.16]{DPS} that $\int_0^1\mu(s,X)ds<\infty$. Referring to \cite[Proposition 3.2.10 (iii) (a)]{DPS} we obtain that 
$$\int_0^\infty \mu(s,X_1)ds=\int_0^\infty \mu(s,X)\chi_{[0,\tau(E^{|X|}(a,\infty)))}(s)ds.$$
By assumption $\tau(E^{|X|}(a,\infty))<\infty$, and therefore, $\int_0^\infty \mu(s,X_1)ds<\infty$. This proves that $X_1\in \cL_1(\cM,\tau)$, as required.

(ii). Since the distribution function is right-continuous inverse of the generalised singular value function (see e.g. \cite[Proposition 3.1.4 (ii)]{DPS}), the assumption $\mu(X)=\mu(Y)$ guarantees that $d_X = d_Y$.
By \cite[Proposition 3.2.10 (iii) (a)]{DPS} we have that $\mu(XE^{|X|}(a,\infty))=\mu(X) \chi_{[0, d_X(a))}$,	$\mu(YE^{|Y|}(a,\infty))=\mu(Y) \chi_{[0, d_Y(a))}$, which implies that $\mu(XE^{|X|}(a,\infty))=\mu(YE^{|Y|}(a,\infty))$. 

Furthermore, since $E(\cM,\tau)\subset S_0(\cM,\tau)$, we have that $d_X(a), d_Y(a)<\infty$ for any $a>0$. Therefore, by \cite[Proposition 3.2.10 (iii) (b)]{DPS} we have that 
$$\mu(t,XE^{|X|}[0,a])=\mu(t+d_X(a),X),\quad  \mu(t,YE^{|Y|}[0,a])=\mu(t+d_Y(a),Y)$$
for all $t\geq 0$. This proves that $\mu(XE^{|X|}[0,a])=\mu(YE^{|Y|}[0,a])$. 
\end{proof}

Next we introduce the main objects of interest in the present paper, traces and symmetric functionals on symmetric quasi-normed spaces. 

\begin{definition}Let $E(\cM,\tau)$ be a symmetric quasi-normed space. A linear functional $\phi$ on $E(\cM,\tau)$ is said to be
	\begin{enumerate}
		\item  a trace, if $\phi(UXU^*)=\phi(X)$ for any $X\in E(\cM,\tau)$ and any unitary $U\in \cM$. 
		\item a symmetric functional, if $\phi(X)=\phi(Y)$ for any $0\leq X,Y\in E(\cM,\tau)$ with $\mu(X)=\mu(Y).$
	\end{enumerate}
\end{definition}

It is clear that any symmetric functional on a symmetric quasi-normed space is a trace. However, in general, there are traces, which are not symmetric functionals (see e.g. \cite[Remark 3.10]{LU}). In the special case, when $\cM$ is a factor (in particular, when $\cM=B(H)$) any trace is necessarily a symmetric functional (see e.g. \cite[Lemma 4.5]{GI1995} and  \cite[Lemma 2.7.4]{LSZ}).  

The following result was proved in \cite[Theorem 4.3.5]{LSZ} for normed spaces. The proof for quasi-normed spaces is a verbatim repetition and, so is omitted. Note, that in \cite{LSZ} symmetric functionals are supposed to be continuous (see \cite[Definition 2.7.3]{LSZ}).
\begin{proposition}\label{lattice_sf}
	The set of all continuous Hermitian symmetric functionals on a symmetric quasi-normed space $E(\cM,\tau)$ is a sublattice of the lattice $E(\cM,\tau)^*$.
\end{proposition}

The following definition introduces important classes of symmetric functionals on symmetric quasi-normed spaces.

\begin{definition}\label{SF}Let $E(\cM,\tau)$ be a symmetric quasi-normed space. A symmetric functional $\phi$ on  $E(\cM,\tau)$ is called
	
	\begin{enumerate}
		\item[(i)] supported at infinity if $\phi(X E^{|X|}(a,+\infty))=0$ for every $a>0$ and every $X \in E(\cM,\tau)$;
		\item[(ii)] supported at zero if $\phi(X E^{|X|}[0,a))=0$ for every $a>0$ and every $X \in E(\cM,\tau)$;
		\item[(iii)] singular if it vanishes on $\mL_1(\cM,\tau) \cap \mathcal M$.
	\end{enumerate}
\end{definition}
Clearly, a functional supported either at zero or at infinity is singular.

The following result was proved in \cite[Theorem 2.10]{DPSS} for fully symmetric quasi-normed functional spaces and fully symmetric functionals on them.

\begin{lemma}\label{decomposition_sf}
Suppose that $\cM$ is atomless or atomic algebra with atoms of equal trace. 	Let $E(\cM,\tau)$ be a symmetric quasi-normed space, such that $E(\cM,\tau)\subset S_0(\cM,\tau)$. For every positive singular symmetric functional $\phi$ on $E(\cM,\tau)$ there exist unique positive functionals $\phi_1, \phi_2$ on $E(\cM,\tau)$ supported at zero and at infinity, respectively, such that $\phi = \phi_1 + \phi_2$.
\end{lemma}

\begin{proof}
Let $\phi$ be a positive singular symmetric functional on $E(\cM,\tau)$ and let $a>0$ be fixed.   Since $E(\cM,\tau)\subset S_0(\cM,\tau)$, it follows that $\tau(E^{|X|}(a,\infty))<\infty$ for any $a>0$ and any $X\in E(\cM,\tau)$. 
By \cref{lem_op_decomp} (i), we have that $X_1=XE^{|X|}(a,\infty)\in E(\cM,\tau)\cap \cL_1(\cM,\tau)$, $X_2=XE^{|X|}[0,a]\in E(\cM,\tau)\cap \cM$ and $X=X_1+X_2$. 
 Suppose that $X=X_1'+X_2'$ is another  decomposition with $X_1'\in  E(\cM,\tau)\cap \cL_1(\cM,\tau)$ and $X_2'\in  E(\cM,\tau)\cap \cM$.  Then
	$X_1-X_1'=X_2'-X_2 \in \mathcal M \cap \mL_1(\cM,\tau)$.  Since $\phi$ is singular, we have that $\phi(X_1-X_1')=\phi(X_2'-X_2) =0$. Hence, we can define 
	$$\phi_1(X) = \phi(X_1),\quad \phi_2(X) = \phi(X_2).$$

It is clear that $\phi_i$ are linear functionals on $E(\cM,\tau)$. Furthermore, if $0\le X, Y \in E(\cM,\tau)$ are such that $\mu(X)=\mu(Y)$, then by \cref{lem_op_decomp} (ii), we have that 
$$\mu(XE^{|X|}[0,a])=\mu(YE^{|Y|}[0,a]),\quad \mu(XE^{|X|}(a,\infty))=\mu(YE^{|Y|}(a,\infty)).$$
Therefore, 
$$\phi_1(X)=\phi(XE^{|X|}(a,\infty))=\phi(YE^{|Y|}(a,\infty))=\phi_1(Y),$$
and similarly, $\phi_2(X)=\phi_2(Y)$. Thus, the functionals $\phi_1$ and $\phi_2$ are symmetric functionals. 

 It is easy to see that $\phi_1$ and $\phi_2$ are positive, supported at zero and at infinity, respectively, and  $\phi = \phi_1 + \phi_2$. 
	It remains to show that the decomposition $\phi=\phi_1+\phi_2$ is unique. Suppose that $\phi=\phi_1'+\phi_2'$  is another such decomposition. Then $\phi_1-\phi_1'=\phi_2'-\phi_2$ is a functional which is supported at zero and infinity simultaneously. Hence, $\phi_1-\phi_1'=\phi_2'-\phi_2=0$, proving that $\phi_1=\phi_1'$ and $\phi_2=\phi_2'$.
\end{proof}

\subsection{Pietsch correspondence}

In this section we recall so-called Pietsch correspondence for symmetric quasi-normed operator spaces from \cite{LU}. In the case when $(\cM,\tau)=(B(H),\tr)$ this recover the standard correspondence introduced by A.Pietsch in \cite{P_trIII}.

Denote by $\eli$ (respectively, $\ell_\infty(\Z_+)$ and $\ell_\infty(\Z_-)$) the algebra of all complex-valued bounded sequences indexed by integers (respectively, by $\Z_+=\{n\in\Z:n\geq 0\}$ and $\Z_-=\{n\in\Z, z< 0\}$) and equipped with the uniform norm $\|\cdot\|_\infty$. Without ambiguity, by $c_0(\Z_+)$ we denote the subsets of $\eli$ and $\ell_\infty(\Z_+)$ of all sequences vanishing at $+\infty$. Similarly, we denote by $c_0(\Z_-)$ the space of all sequences vanishing at $-\infty$.

Let $S(\Z)$ be the algebra of all two-sided sequences bounded at $+\infty$, that is, all sequences $\{x_n\}_{n\in\Z}$ such that  $\sup_{k\geq n}|x_k|<\infty$ for all $n\in\Z$. Similarly, we can define the spaces $S(\Z_-)$ and $S(\Z_+)$. It is clear that $S(\Z_+)=\ell_\infty(\Z_+).$

Recall from \cite[Section 2]{LU} (see \cite{P_trIII} for sequences indexed by $\Z_+$)
, that for $x\in S(\Z)$ the ordering numbers $o_n(x)$, $n\in\Z$ are defined by setting
\begin{equation}
	o_n(x)=\sup_{k\geq n}|x_k|,\quad n\in\Z.
\end{equation}
If $x\in S(\Z)$ is a positive decreasing sequence, then $o(x)=x$.

On the space $S(\Z)$ we define the right- and left-shift operators $S_+, S_-:S(\Z)\to S(\Z)$, by setting
$$(S_+x)_n=x_{n-1}, \ (S_-x)_n=x_{n+1}, \quad x=\{x_n\}_{n\in\Z}\in S(\Z).$$
Similarly for the spaces $S(\Z_+)$ and $S(\Z_-)$ we define the corresponding shift operators as follows 
$$S_+\{x_n\}_{n\in\Z_+}=\{0, x_1,x_2,\dots\}, \quad S_+\{x_n\}_{n\in\Z_-}=\{\dots, x_{n-1}, \dots, x_{-3}, x_{-2}\},$$
$$S_-\{x_n\}_{n\in\Z_+}=\{x_2,x_3,\dots\}, \quad S_-\{x_n\}_{n\in\Z_-}=\{\dots, x_{n-1}, \dots, x_{-2}, x_{-1}, 0\}.$$

\begin{definition}\label{def_si}
	A linear subspace $E(\Z)\subset S(\Z)$ is called a \emph{shift-monotone space} if
	\begin{enumerate}
		\item $x\in S(\Z), y\in E(\Z)$ and $o(x)\leq o(y)$ implies that $x\in E(\Z).$ 
		\item $S_+x\in E(\Z)$ for any $x\in E(\Z)$. 
	\end{enumerate}
	Similarly, one can define shift-monotone spaces on $\Z_-$ and on $\Z_+$ (in the latter case, one recovers the original definition due to A. Pietsch \cite{P_trI}).
\end{definition}


It is clear that $o(S_-x)\leq o(x)$ for any $x\in S(\Z)$. In particular, any shift-monotone space is invariant with respect to both right- and left-shift operators.

Let $\cM$ be a nonatomic (or atomic with atoms of equal trace) von Neumann algebra equipped with a faithful normal semifinite trace $\tau$ and let $0\leq A\in S_0(\cM,\tau)$. By \cite[Theorem 2.3.11]{LSZ}) (see also  \cite[Theorem 3.5]{LU}) that there exists a $\sigma$-finite commutative subalgebra $\cM_0$ of $\cM$ such that the restriction $\tau|_{\cM_0}$ is semifinite and the analogue of diagonal operator exists. Namely, there exists a trace preserving $*$-isomorphism $\iota$ from $S(0,\tau(\mathbf{1}))$ onto $S(\cM_0,\tau|_{\cM_0})$ (respectively, $\ell_\infty\to \cM_0$), such
that $\mu(\iota(X))=\mu(X)$ for any $X\in S(0,\tau(\mathbf{1}))$ and $\iota(\mu(A))=A.$

Define the operators $D:S(\Z)\to S(0,\infty)$ and $\cD:S(\Z)\to S(\cM,\tau)$ by the formulae
\begin{equation}\label{def_D}
	Dx=\sum_{n\in\Z} x_n\chi_{[2^n,2^{n+1})} \ \text{and} \ \cD x=\iota Dx,\quad x\in S(\Z).
\end{equation}

It was proved in \cite[Theorems 3.7 and 4.8]{LU} that (under certain assumptions on $(\cM,\tau)$) there is a bijective correspondence between symmetric quasi-normed spaces on $\cM$ and shift-monotone sequence spaces in $S(\Z)$, as specified in the theorem below.

\begin{theorem} 	Let $\cM$ be a nonatomic von Neumann algebra with $\tau(\mathbf{1})=\infty$. The rule
\begin{align*}
E(\Z)&:=\{x\in S(\Z): \cD x\in E(\cM,\tau)\}\\
&=\{x\in S(\Z): (Dx)^*=\mu(X) \text{ for some } X\in E(\cM,\tau)\},\\
\|x\|_{E(\Z)}&:=\|\cD x\|_{E(\cM,\tau)}, \quad x\in E(\Z),
\end{align*}
and 
\begin{align*}
E(\cM,\tau)&:=\{X\in S(\cM,\tau): \{\mu(2^n,X)\}_{n\in \Z} E(\Z)\},\\
\|X\|_{E(\cM,\tau)}&:=\|\{\mu(2^n,X)\}\|_{E(\Z)},\quad X\in E(\cM,\tau),
\end{align*}
defines a bijective correspondence $E(\Z)\leftrightarrows E(\cM,\tau)$ between symmetric quasi-normed spaces on $\cM$ and quasi-normed shift-monotone sequence spaces in $S(\Z)$. 

If $\cM$ is nonatomic algebra with $\tau(\mathbf{1})=1$ (or atomic algebra with atoms of equal trace), the correspondence is preserved when $\Z$ is replaced by $\Z_-$ (respectively, by $\Z_+)$. 
\end{theorem}

We shall need the following definition from \cite[Definition 2.5]{LU}. 
\begin{definition}Let $E(\Z)$ be a shift-monotone sequence space and let $X\in \smt$. An infinite series representation 
	$$X=\sum_{k\in\Z} X_k,$$
	where the convergence is understood in the measure topology, 
	is called an $E(\Z)$-dyadic representation of $X$ if $X_k\in \cM$, $\tau(s(X_k))\leq 2^k$ and $$\Big\{\big\| X-\sum_{k=-\infty}^nX_k\big\|\Big\}_{n\in\Z}\in E(\Z).$$
\end{definition}

For the notion of the measure topology we refer to \cite[Section 2]{Ne}.

By \cite[Lemma 2.10]{LU} an operator $X\in \smt$ belongs to $E(\cM,\tau)$ if and only if there exists an $E(\Z)$-dyadic representation of $X$. 

In the following example we specify the shift-monotone sequence space  corresponding to the symmetric space $\mL_g(\cM,\tau)$.

\begin{example}	Let $\cM$ be a nonatomic von Neumann algebra with $\tau(\mathbf{1})=\infty$. The space $\mL_g(\cM,\tau)$ corresponds to a shift-monotone sequence space of the following form:
	$$\ell_g(\Z)=\left\{x\in S(\Z): \sup_{n\in \Z} \frac{o_n(x)}{g(2^n)} < \infty\right\}. $$
	
Indeed, by \cite[Theorem 3.7]{LU} the corresponding shift-monotone sequence space is defined by
\begin{align*}
	E(\Z) &= \left\{ x\in S(\Z) : (Dx)^* = \mu(X) \ \text{for some} \ X\in \mL_g(\cM,\tau)\right\}\\
	& = \left\{ x\in S(\Z) : \sup_{t>0} \frac{(Dx)^*(t)}{g(t)} < \infty\right\}\\
	&=\left\{x\in S(\Z): \sup_{n\in \Z} \frac{o_n(x)}{g(2^n)} < \infty\right\},
\end{align*}
since by \cite[Lemma 3.1]{LU} $o_n(x)= (Dx)^*(2^n)$, $n\in \Z$ and $g$ satisfies \eqref{g}.

If $\cM$ is a nonatomic with $\tau(\mathbf{1})=1$ (or atomic with atoms of equal trace) the proof stays the same after replacing $\Z$ by $\Z_-$ (respectively, $\Z_+$).
\end{example}

\subsection{Linear functionals on $\ell_\infty(\Z)$}
 
We now collect some of the properties of linear functionals on $\ell_\infty(\Z)$, which shall be used in the sequel.

As before, we denote by $S_\pm$ the right and left shift operators on $\ell_\infty(\Z)$.

\begin{definition} A linear functional $\gamma$ on $\ell_\infty(\Z)$ is said to be 
\begin{enumerate}
\item $S_+$-invariant (respectively, $S_-$-invariant) if
		$\gamma(S_+x)=\gamma(x)$ (respectively, $\gamma(S_-x)=\gamma(x)$) for every $x\in \ell_\infty(\Z)$;

\item supported at $+\infty$ (respectively, $-\infty$) if $\gamma(\chi_{(-\infty, a)})=0$ (respectively,\\ $\gamma(\chi_{(a, +\infty)})=0$) for every $a\in\Z$;

\item singular if it vanishes on sequences with finite support.
\end{enumerate}

\end{definition}

Similar definitions are applied for $E(\Z)$, $\ell_\infty(\Z_+)$ and $\ell_\infty(\Z_-)$.
It follows from \cite[Lemma 8.1]{P_PDO} that a linear functional $\gamma$ on $\ell_\infty(\Z_+)$ is $S_+$-invariant if and only if it is $S_-$-invariant. This equivalence also holds on $\ell_\infty(\Z_-)$.

The following result was proved in \cite[Lemma 4.8]{SSUZ} in the case of real-valued sequences on $\Z_+$. The proof for complex-valued sequences on $\Z$ is the same and, so omitted.

\begin{proposition}\label{lattice_si}
	The set of all continuous Hermitian shift-invariant functionals on $\ell_\infty(\Z)$ is a sublattice of the lattice $\ell_\infty(\Z)^*$.
\end{proposition}

\begin{remark}\label{supp_vanish}
	A shift-invariant functional on $\ell_\infty(\Z_+)$ vanishes on any finitely supported sequence. If additionally this functional is continuous, then it vanishes on $c_0(\Z_+)$. Similar statements hold for $\ell_\infty(\Z_-)$. None of these statements hold on $\ell_\infty(\Z)$. However, any continuous shift-invariant functional on $\ell_\infty(\Z)$ supported at $+\infty$ (respectively, $-\infty$) vanishes on $c_0(\Z_+)$ (respectively, $c_0(\Z_-)$).
\end{remark}

\begin{lemma}\label{decomposition_si}
	For every positive singular shift-invarinant functional $\gamma$ on $\ell_\infty(\Z)$ there exist unique positive shift-invarinant functionals $\gamma_1, \gamma_2$ on $\ell_\infty(\Z)$ supported at $+\infty$ and at $-\infty$, respectively, such that $\gamma = \gamma_1 + \gamma_2$.
\end{lemma}

\begin{proof}
	For a fixed $a\in \Z$  set $\gamma_1(x)= \gamma(x\chi_{(a,+\infty)})$ and $\gamma_2(x)=\gamma(x\chi_{(-\infty, a)})$ for $x\in \ell_\infty(\Z)$. Since $\gamma$ is singular, the representation $\gamma = \gamma_1 + \gamma_2$ holds and is unique. Clearly, $\gamma_1, \gamma_2$ are positive and supported at $+\infty$ and at $-\infty$, respectively.
\end{proof}

\begin{definition} A bounded linear functional $\omega$ on $\ell_\infty(\Z_+)$ is said to be an 
	extended limit at $+\infty$ if $\omega(x)= \lim_{n\to +\infty} x_n$ for every convergent sequence $x$. 
\end{definition}

The following result lists several simple properties of extended limits, which will be used in the paper.

\begin{proposition}\label{EL}
	(i) For every real-valued $x=\in \ell_\infty(\Z_+)$ one has
	$$\{\omega(x) \ : \ \omega \ \text{is an extended limit}\} =[\liminf_{n\to+\infty} x_n, \limsup_{n\to+\infty} x_n].$$
	
	(ii) Every extended limit is a positive (and, so, continuous) functional on $\ell_\infty(\Z_+)$.
	
	(iii) For every $x, y\in \ell_\infty(\Z_+)$ such that $y_n \to a$ as $n\to+\infty$ and every extended limit $\omega$ on $\ell_\infty(\Z_+)$ we have
	$$\omega (xy)=a\omega(x).$$
\end{proposition}

\begin{proof}
	Part (i) is a direct consequence of Hahn-Banach theorem. Part (ii) follows from Part (i). Part (iii) follows from the fact that every extended limit vanishes on finitely supported sequences.
\end{proof}

We shall introduce a subset of all extended limits.
\begin{definition} A linear functional $B$ on $\ell_\infty(\Z_+)$ is said to be a Banach limit 
	if
	\begin{enumerate}
		\item[(i)]
		$B\geqslant0$, that is $Bx\geqslant0$ for every $x\geqslant0$,
		
		\item[(ii)]
		$B(1,1,\ldots)=1$,
		
		\item[(iii)]
		$B(S_+x)=B(x)$ for every $x\in \ell_\infty(\Z_+)$.
	\end{enumerate}
\end{definition}

A sequence $x\in \ell_\infty(\Z_+)$ is called almost convergent if all Banach limits coincide on $x$ (see~\cite{L}).

%
%
%
%
%

\subsection{Simple lemmas}
\begin{lemma}\label{aconv1}
	Let $g : (0,\infty) \to (0,\infty)$ be a positive decreasing function such that 
	\begin{equation}\label{limcond}
		\lim_{t\to +\infty} \frac{g(t)}{g(2t)} =2 \ \left(\text{respectively,} \ \lim_{t\to 0+} \frac{g(t)}{g(2t)} =2 \right).
	\end{equation}
	If $f \in L_\infty(0,\infty)$ is such that $\int_{0}^{2^{n}} f(s) \ ds =O(2^n g(2^n))$ for every $n\in \Z$, then the sequence
	$$\left\{ \frac{1}{2^n g(2^n)}\int_{2^n}^{2^{n+1}} f(s) \ ds \right\}_{n \in \Z}$$
	belongs to $\Ra +c_0(\Z_+)$ (respectively, $\Ra +c_0(\Z_-)$).
\end{lemma}

\begin{proof}
	Define the sequence $y=\{y_n\}_{n\in \Z}$ by setting $y_n= \frac{1}{2^n g(2^n)}\int_{0}^{2^{n}} f(s) \ ds$. By assumption $y\in \ell_\infty(\Z)$. Further
	\begin{align*}
		\frac{1}{2^n g(2^n)}\int_{2^n}^{2^{n+1}} f(s) \ ds  &=\frac{1}{2^n g(2^n)}\int_0^{2^{n+1}} f(s) \ ds-\frac{1}{2^n g(2^n)}\int_0^{2^{n}} f(s) \ ds\\
		&= y_{n+1}-y_n +\left(\frac{1}{2^n g(2^n)} - \frac{1}{2^{n+1}g(2^{n+1})}\right)\int_0^{2^{n+1}} f(s) \ ds\\
		&= y_{n+1}-y_n +\left(2 \frac{g(2^{n+1})}{g(2^n)}-1 \right) y_{n+1}.
	\end{align*}
	
	Since $y \in \ell_\infty(\Z)$ and $g$ satisfies condition \eqref{limcond}, the last term above tends to zero as $n\to+\infty$ (respectively, as $n\to-\infty$).
	This proves the assertion.
\end{proof}

\begin{lemma}\label{simple}
	For every $X\in \Lw(\cM,\tau)$ we have 
	$$\int_{2^n}^{2^{n+1}}\mu(s,X)ds = O\left(2^n g(2^n) \right), \ n\in \Z.$$
\end{lemma}

\begin{proof} 
	Since $\mu(X)$ is decreasing, it follows that 
	\begin{align*}\frac{1}{2^n g(2^n)}\left| \int_{2^n}^{2^{n+1}}\mu(s,X)ds \right| &\le \frac{1}{2^n g(2^n)} 2^n \mu(2^n,X) \\
		&\le \frac{g(2^n-1)}{g(2^n)}\nw{X} \le C\nw{X}, \ n\in \Z,
	\end{align*}
	since $\frac{g(t)}{g(2t)} \le C$ by~\eqref{g}.
\end{proof}

\begin{lemma}\label{lem_l1}
A bounded positive decreasing function $g \in L_\infty(0,\infty)$ belongs to $L_1(0,\infty)$ if and only if the sequence
$\{2^n g(2^n) \}_{n=0}^\infty$ belongs to the space $\ell_1$ of all summable sequences.
\end{lemma}

\begin{proof}
Since $g$ is decreasing, for every $m\in \N$ we have
\begin{align*}
	\frac12 \sum_{n=0}^{m-1} 2^{n+1} g(2^{n+1})&=\sum_{n=0}^{m-1} \sum_{i=2^n+1}^{2^{n+1}} g(2^{n+1}) \le \sum_{i=2}^{2^m} g(i)\\ &\le \sum_{n=0}^{m-1} \sum_{i=2^n+1}^{2^{n+1}} g(2^n) 
	= \sum_{n=0}^{m-1} 2^n g(2^n).
\end{align*}

Therefore, the series $\sum_{i=2}^\infty g(i)$ converges if and only if the series $\sum_{n=0}^{\infty} 2^n g(2^n)$ converges. Since $g$ is bounded positive and decreasing, then 
$$\sum_{i=3}^\infty g(i) \le \int_2^\infty g(s) ds\le \sum_{i=2}^\infty g(i).$$
 Hence, the claim follows. 
\end{proof}

\section{CONSTRUCTION OF TRACES}
In this section we modify the dyadic approach to obtain a bijective correspondence between all continuous symmetric functionals on spaces $\Lw(\cM,\tau)$ and all continuous $S_+$-invariant functionals on a single space $\eli$.

The following correspondences were established in \cite[Theorems 3.7 and 3.9 and Corollary 3.13 ]{LU}. We will state it for spaces $\Lw(\mathcal M, \tau)$.

\begin{theorem}\label{Bijective correspondence}
Let $\cM$ be a nonatomic (or atomic with atoms of equal trace) von Neumann algebra equipped with a faithful normal semifinite trace $\tau$. The map
		\begin{align*}
			\theta&\mapsto \phi,\\
			\phi(X)&=\theta(\{{2^{-k}}\tau(X_k)\}_{k\in\Z}), \quad X\in \Lw(\cM,\tau),
		\end{align*}
		where $X=\sum_{k\in\Z} X_k$ is an $\lw(\Z)$-dyadic representation of $X\in \Lw(\cM,\tau)$, gives a bijective correspondence between all symmetric functionals on $\Lw(\cM,\tau)$ and $\frac12 S_+$-invariant functionals $\theta$ on $\lw(\Z)$. The inverse of this rule  is given by 
		$$\theta(x)=\phi(\cD x), \quad x\in \lw(\Z).$$
\end{theorem}

The idea of the following lemma is due to A. Pietsch \cite[Lemma 6.1]{Pietsch_simply_gen_more}. We prove it in  the case of double sequences.

\begin{lemma}\label{the map of anger}
	Let $g : (0,\infty) \to (0,\infty)$ be a positive decreasing function. 
	
	(i) The map $N : \ell_\infty(\Z) \to \ell_g(\Z)$, given by $(Nx)_n = g(2^n) x_n$, $n\in \Z$ is an isometry;
	
	(ii) 
If 
	$$ \lim_{t\to +\infty} \frac{g(t)}{g(2t)} =2\ \left( \text{respectively,} \lim_{t\to 0} \frac{g(t)}{g(2t)} =2\right),$$
	then the dual map $N^* : \ell_g(\Z)^* \to \ell_\infty(\Z)^*$ is an isometry between the classes of all continuous $\frac12 S_+$-invariant functionals on $\ell_g(\Z)$ supported at $+\infty$ (respectively, supported at $-\infty$) and that of all continuous $S_+$-invariant functionals on $\ell_\infty(\Z)$ supported at $+\infty$ (respectively, supported at $-\infty$).
\end{lemma}

\begin{proof}
	The fact that $N$ and $N^*$ are isometries is straightforward.
	
	Let $\theta \in \ell_g(\Z)^*$ be a continuous $\frac12 S_+$-invariant functional supported at $+\infty$ (respectively, supported at $-\infty$). Clearly, $N^*\theta$ is a continuous functional supported at $+\infty$ (respectively, supported at $-\infty$). Further,
	\begin{align*}
		(N^*\theta)(S_+x)&=(\theta \circ N)(S_+x)=\theta(\{g(2^n) x_{n-1}\}_{n\in \Z})\\
		&=\theta\left(\left\{\frac12g(2^{n-1})  x_{n-1} \cdot 2\frac{g(2^n)}{g(2^{n-1})}\right\}_{n\in \Z}\right)\\
		&=\theta\left(\frac12 S_+\left\{g(2^{n})  x_{n} \right\}_{n\in \Z}\right)
		-\theta\left(\frac12 S_+\left\{g(2^{n})  x_{n} \right\}_{n\in \Z} \left(1-2\frac{g(2^n)}{g(2^{n-1})} \right)\right).
	\end{align*}
	Using assumptions on $g$ and \cref{supp_vanish}, we deduce that the second summand equal to zero. Thus, $(N^*\theta)(S_+x)=(N^*\theta)(x),$
	that is, $N^*\theta$ is $S_+$-invariant functionals on $\ell_\infty(\Z)$.
	
	Let $\gamma \in \ell_\infty(\Z)^*$ be a continuous $S_+$-invariant functional supported at $+\infty$ (respectively, supported at $-\infty$). Clearly, $(N^*)^{-1}\gamma$ is a continuous functional supported at $+\infty$ (respectively, supported at $-\infty$). Further,
	\begin{align*}
		((N^*)^{-1}\gamma)(\frac12 S_+x)&=(\gamma \circ N^{-1})(\frac12 S_+x)=\gamma(\{\frac1{g(2^n)} \frac12 x_{n-1}\}_{n\in \Z})\\
		&=\gamma\left(\left\{ \frac{x_{n-1}}{g(2^{n-1})} \cdot \frac12\frac{g(2^{n-1})}{g(2^n)}\right\}_{n\in \Z}\right)\\
		&=\gamma\left(S_+\left\{\frac{x_{n}}{g(2^{n})} \right\}_{n\in \Z}\right)=(N^*\gamma)(x),
	\end{align*}
	that is, $(N^*)^{-1}\gamma$ is $\frac12 S_+$-invariant functionals on $\ell_g(\Z)$.
\end{proof}

\begin{remark}
	Condition \eqref{limcond} is essential. Let $x_n = g(2^n)$. The sequence
	$$\frac12 \frac{x_{n-1}}{g(2^n)}  - \frac{x_{n-1}}{g(2^{n-1})}= \frac12 - \frac{g(2^n)}{g(2^{n-1})}$$
	does not vanish, in general. So, it does not belong to
	$(I-S_+)(\ell_\infty(\Z))$.
	
	Continuity is essential as well. Let $g(t)=\frac{\log_2 t}{t}$ and $x_n = g(2^n)$. The sequence
	$$\frac12 \frac{x_{n-1}}{g(2^n)}  - \frac{x_{n-1}}{g(2^{n-1})}=\frac12 (1- \frac{n}{n-1}) = \frac{1/2}{n-1}$$
	does not belong to $(I-S_+)(\ell_\infty(\Z))$.

	 Hence, in both cases there exists an $S_+$-invariant functionals on $\ell_\infty(\Z)$ which does not vanish on this sequence, and, so, $(N^*)^{-1}\gamma$ is not $\frac12 S_+$-invariant.
\end{remark}

Define operators $D_g:S(\Z)\to L_g(0,\infty)$ and $\cD_g:S(\Z)\to \Lw(\cM,\tau)$ by setting
\begin{equation}\label{def_Dg}
	D_gx=\sum_{n\in\Z} x_n g(2^n) \chi_{[2^n,2^{n+1})} \ \text{and} \ \cD_g x=\iota D_gx,\quad x\in S(\Z).
\end{equation}

For a given $X\in \Lw(\cM,\tau)$ we define a sequence $\Phi_g X$ by setting
\begin{equation}\label{def_Phi_g}
\Phi_g(X):=\left\{\frac1{2^n g(2^n)}\int_{2^n}^{2^{n+1}}\mu(s,X)ds\right\}_{n\in\Z}
\end{equation}
By \cref{simple}, the operator $\Phi_g : \Lw(\cM,\tau) \to \ell_\infty(\Z)$ is well-defined and bounded.

\begin{proposition}\label{inv1}
Let $g : (0,\infty) \to (0,\infty)$ be a positive decreasing function.

(i) Equality 
$$\phi(X)=\phi(\cD_g\Phi_g X), \quad X\geq 0$$
holds for any symmetric functional $\phi$ on  $\Lw(\cM,\tau)$;

(ii) Equality 
$$\theta(x) = \theta(\Phi_g \cD_g x), \ x \ge 0$$
holds for every continuous $S_+$-invariant functional $\theta$ on $\ell_\infty(\Z)$ supported at $+\infty$ (respectively, supported at $-\infty$) if
$$ \lim_{t\to +\infty} \frac{g(t)}{g(2t)} =2\ \left( \text{respectively,} \lim_{t\to 0} \frac{g(t)}{g(2t)} =2\right) .$$
\end{proposition}

\begin{proof}
	Part (i) is proved in \cite[Lemma 3.12]{LU}. 
	
(ii) For every $0 \le x \in \ell_\infty(\Z)$ and $n\in \Z$ we have
$$(x - \Phi_g \cD_g x)_n = x_n - \frac{1}{2^n g(2^n)} \int_{2^n}^{2^{n+1}}\mu(s,\cD_g x)ds.$$

By definition of $D_gx$ we have that
$\int_{2^n}^{2^{n+1}} (D_gx)(s) ds =2^n g(2^n) x_n, \ n\in \Z$
and therefore
\begin{equation}\label{e0}(x - \Phi_g \cD_g x)_n = \frac{1}{2^n g(2^n)} \int_{2^n}^{2^{n+1}} \left((D_gx)(s) -\mu(s,\cD_gx) \right) ds. 
\end{equation}

By \cite[Theorem 4.5]{LU} $\mu(\cD_gx) = \mu(D_gx)$ for every $x \in S(\Z)$. Further, it follows from \cite[Lemma 2.2]{FigielKalton}
$$\left|\int_{0}^{2^{n+1}} \left((D_gx)(s) -\mu(s, D_gx) \right) ds\right| \le 8 c \cdot  2^{n+1} g(2^{n+1}),$$
for every $n\in \Z$ and some constant $c>0$.
The assertion follows from \cref{aconv1}.
\end{proof}

The following theorem describes the bijective correspondence between symmetric functionals on $\mL_g(\cM,\tau)$ and shift-invariant functionals on $\ell_\infty(\Z)$ in the continuous case.

\begin{theorem}\label{one-to-one1}
Assume that $\cM$ is an atomless  von Neumann algebra with $\tau(\mathbf{1})=1$.
Let $g : (0,\infty) \to (0,\infty)$ be a positive decreasing function satisfying condition \eqref{limcond}. 
Let $\cM$ be a nonatomic (or atomic with atoms of equal trace) von Neumann algebra. The rules 
$$\phi(X)=\theta\left(\{\frac{\tau(X_k)}{2^{k} g(2^k)}\}_{k\in\Z}\right), \quad X\in \mL_g(\cM,\tau),$$
$$\theta(x)=\phi(\cD_g x), \quad x\in \ell_\infty(\Z). $$
where $X=\sum_{k\in\Z} X_k$ is a dyadic $\ell_g(\Z)$-representation of $X\in \mL_g(\cM,\tau)$, give a bijective correspondence between all continuous symmetric functionals on $\mL_g(\cM,\tau)$ and all continuous $S_+$-invariant functionals on $\ell_\infty(\Z)$.

In the case, when $\cM$ is an atomless von Neumann algebra with $\tau(\mathbf{1})=$ (respectively, an atomic von Neumann algebra with atoms of equal trace), the similar assertion holds assuming the asymptotic behaviour in \eqref{limcond} only at $0$ (respectively, at $\infty$). 
\end{theorem}

\begin{proof}
	From \cref{Bijective correspondence} and \cite[Remark 3.16]{LU} it follows (in particular) that there is a bijection between all positive symmetric functionals on $\mL_g(\cM,\tau)$ supported at $+\infty$ (respectively, supported at zero) and all positive $\frac12 S_+$-invariant functionals on $\lw(\Z)$ supported at $+\infty$ (supported at $-\infty$). Combining this with \cref{the map of anger} we obtain a bijection between all positive symmetric functionals on $\mL_g(\cM,\tau)$ supported at $+\infty$ (respectively, supported at zero) and all positive $S_+$-invariant functionals on $\ell_\infty(\Z)$ supported at $+\infty$ (supported at $-\infty$). 
	
	It follows from \cref{decomposition_sf} that every positive symmetric functional on $\mL_g(\cM,\tau)$ admits a unique decomposition into a sum of positive symmetric functionals on $\mL_g(\cM,\tau)$ supported at $+\infty$ and that supported at zero. By \cref{decomposition_si} similar result holds for positive $S_+$-invariant functionals on $\ell_\infty(\Z)$. 
	Therefore, there is a bijection between all positive symmetric functionals on $\mL_g(\cM,\tau)$ and all positive $S_+$-invariant functionals on $\ell_\infty(\Z)$.
	
	By \cref{lattice_si} the set of all continuous Hermitian $S_+$-invariant functionals on $\ell_\infty(\Z)$ form a sublattice in the lattice $\ell_\infty^*$ and by \cref{lattice_sf} the set of all continuous Hermitian symmetric functionals on $\mL_g(\cM,\tau)$ is a lattice too.
	Hence, there is a bijection between all continuous Hermitian symmetric functionals on $\mL_g(\cM,\tau)$ and all continuous Hermitian $S_+$-invariant functionals on $\ell_\infty(\Z)$.
	
	By Jordan decomposition every $S_+$-invariant functionals on $\ell_\infty(\Z)$ and every symmetric functionals on $\mL_g(\cM,\tau)$ are uniquely written as a linear combination of Hermitian functionals. If the original functional is continuous, then the corresponding Hermitian functionals continuous too. This proves the assertion.	
\end{proof}


\begin{corollary}\label{bij}
Let $g, f : (0,\infty) \to (0,\infty)$ be positive decreasing functions satisfying condition \eqref{limcond}. 
There is a bijection $i_{g,f}$ between the set of all continuous symmetric functionals on $\mL_g(\cM,\tau)$ and $\mathcal L_{f}(\cM,\tau),$
given by 
\begin{equation}\label{i}
	i_{g,f}(\phi) = \phi \circ \cD_g \circ \Phi_f.
\end{equation}
\end{corollary}

\begin{proof}
It follows from \cref{one-to-one1} that every continuous symmetric functional $\phi_1$ on $\Lw(\cM,\tau)$ corresponds to a continuous $S_+$-invariant linear functional $\theta = \phi_1 \circ \cD_g$ on $\ell_\infty(\Z)$. Further, \cref{one-to-one1} yields that the functional $\phi_2 = \theta \circ \Phi_f = \phi_1 \circ \cD_g \circ \Phi_f$ is a continuous symmetric functional on $\mathcal L_{f}(\cM,\tau).$

Applying \cref{one-to-one1} to the continuous symmetric functional $\phi_2$ on $\mathcal L_{f}(\cM,\tau)$ in the similar way, we obtain that it corresponds to a symmetric functional $\phi = \phi_2 \circ \cD_f \circ \Phi_g$ on $\Lw(\cM,\tau)$. Hence,
$$\phi = \phi_1 \circ \cD_g \circ \Phi_f \circ \cD_f \circ \Phi_g.$$

Since $\phi_1 \circ \cD_g$ is a continuous $S_+$-invariant linear functional on $\ell_\infty(\Z)$, it follows from  \cref{inv1}(ii) that 
$$\phi = \phi_1 \circ \cD_g \circ \Phi_g.$$

Since $\phi_1$ is a symmetric functional on $\Lw(\cM,\tau)$, it follows from \cref{inv1}(i) that 
$$\phi = \phi_1.$$
\end{proof}

\begin{corollary}\label{bijpos}
Let $g, f$ be as in \cref{bij}. Since the operators $\cD_g$ and $\Phi_f$ are positive, it follows that the mapping $i_{g,f}$ given by~\eqref{i} is a bijection between the set of all \textbf{positive} symmetric functionals on $\Lw(\cM,\tau)$ and $\mathcal L_{f}(\cM,\tau).$
\end{corollary}

We conclude this section with the discussion of  measurability with respect to all normalised continuous symmetric functionals. We call a symmetric functional $\phi$ on $\Lw(\cM,\tau)$ normalised if $\phi(\cD_g\chi_{\Z})=1$.

%

\begin{proposition}\label{NCT} Let $g : (0,\infty) \to (0,\infty)$ be a positive decreasing function satisfying condition \eqref{limcond}.
All normalised continuous symmetric functionals on $\Lw(\cM,\tau)$ take the same value on an operator $X\in \Lw(\cM,\tau)$ if and only if all normalised positive symmetric functionals on $\Lw(\cM,\tau)$ take the same value on $X$.
\end{proposition}

\begin{proof} Since every positive functional is continuous, the ``only if'' part is straightforward.

Conversely, let $X\in \Lw(\cM,\tau)$ and all normalised positive symmetric functionals on $\Lw$ equal to $a\in \C$ on $X$. 

By \cref{lattice_sf} the set of all continuous Hermitian symmetric functionals on $\Lw(\cM,\tau)$ is a lattice, and therefore it follows from~\cite[Theorem 1.1.1 (ii)]{MN} that every continuous symmetric functional $\phi$ on $\Lw(\cM,\tau)$ can be written in the form $\phi = \phi_+ - \phi_-$, where $\phi_+$, $\phi_-$ are positive symmetric functionals on $\Lw(\cM,\tau)$.

If we further assume that $\phi$ is normalised, then $$\phi_+(\cD_g\chi_{\Z}) - \phi_-(\cD_g\chi_{\Z})=1.$$ Note that the functionals $\phi_{\pm}:= \frac{\phi_\pm}{\phi_\pm(\cD_g\chi_{\Z})}$ are normalised positive symmetric functionals on $\Lw(\cM,\tau)$ and
$$\phi = \alpha_1 \phi_+ + (1-\alpha_1) \phi_-,$$
where $\alpha_1=\phi_+(\cD_g\chi_{\Z}).$
Hence,
\begin{align*}
	\phi(X) &= \alpha_1 \phi_+(X) + (1-\alpha_1) \phi_-(X)=a,
\end{align*}
that is all normalised Hermitian continuous symmetric functionals on $\Lw(\cM,\tau)$ take the same value on an operator $X$. The assertion follows from the Jordan decomposition.
\end{proof}

We finish this section with an example of a magnetic Schr\"odinger operator which falls into one of simply generated spaces considered in this paper.

\begin{example}
	Let $\mathbb R^d = \mathbb R^n \times \mathbb R^m$ and we write a variable $z \in \mathbb R^d$ by $z = (x, y) \in \mathbb R^n_x \times \mathbb R^m_y$.
	We consider the operator:
	$$H(A, V ) =\frac12 \left((i\nabla_{(x,y)} + A(x, y))^2 + V (x, y) \right),$$
	where $i =\sqrt{-1}$ and $\nabla_{(x,y)}$ denotes the gradient operator. Assume that $A$ and $V$ are such that
	\begin{enumerate}
		\item[(i)] $V (x, y) \in C^1(\mathbb R^n \times \mathbb R^m)$ is a real valued function;
		
		\item[(ii)] There exist positive constants $p, q$ and $C \ge c > 0$ such that
		$$c (1 + |x|^2)^p |y|^{2q} \le V (x, y) \le C (1 + |x|^2)^p |y|^{2q}$$ for all $(x, y) \in \mathbb R^n \times \mathbb R^m$;
		
		\item[(iii)] $A(x, y) = (a_1(x, y), \dots , a_d(x, y)) \in C^2(\mathbb R^d, \mathbb R^d)$;
		
		\item[(iv)] There exist constants $a, b$ satisfying $0 \le a < p, 0 \le b < q, (q + 1)a <p(b + 1)$ and $C_1 > 0$ such that for every $j = 1, 2, \dots, d$ and $|\alpha| \le 2$,
		$$\left|\partial^\alpha_{x,y} a_j(x, y) \right| \le C_1 (1 + |x|^2)^a |y|^{2b}.$$
	\end{enumerate}
	
	For $pm=qn$ and $m(1+q+p) = 2q$ it follows from \cite[Corollary 2.3]{Aramaki2001} that there exist $c_3>0$ such that
	$$\lim_{t\to\infty} \frac{N_{H(A,V)}(t)}{t\log t} = c_3,$$
	where $N_{H(A,V)}(t) = \left|\left\{ n\ge0 : \lambda(n, H(A,V)) \le t\right\}\right|$ is the eigenvalue counting function. Further, for the distribution function we have 
	$$1 = \lim_{s\to0} \frac{d_{(1+H(A,V))^{-1}}(s)}{N_{H(A,V)}(1/s)} = \lim_{s\to0} \frac{d_{(1+H(A,V))^{-1}}(s)}{c_3 1/s \log(1/s)}.$$
	Taking the asymptotic inverses of both the numerator and the denominator, we obtain
	$$\lim_{t\to\infty} \frac{\mu(t, (1+H(A,V))^{-1})}{\log(t+2)/(t+2)} = c_3.$$
	
	This implies that the operator $(1+H(A,V))^{-1}$ belongs to the space $\mathcal L_g$ with $g(t)= \log(t+2)/(t+2).$ 
	Moreover, it follows from \cite[Proposition 5.12]{Ponge} that all positive normalised traces on $\mathcal L_g$ equal $c_3$ on this operator. 
	Hence, by \cref{NCT} all continuous normalised traces on $\mathcal L_g$ equal $c_3$ on this operator.
\end{example}

\section{DIXMIER TRACES}\label{sec:Dix}

In this section we characterise $S_+$-invariant functionals corresponding to Dixmier traces in \cref{one-to-one1}.
We shall also show that the mapping~\eqref{i} is not a bijection between the set of all Dixmier traces on distinct simply generated symmetric spaces.
We restrict our attention to the case when $(\mM, \tau)$ is the algebra $B(H)$ with the standard trace. In this case, symmetric functionals on $\Lw(H) := \Lw(B(H), \rm Tr)$ correspond to $S_+$-invariant functionals on the space $\ell_\infty :=\ell_\infty(\Z_+)$ of one-sided sequences and the classes of traces and that of symmetric functionals coincide. Furthermore,  the definition of $\Lw(H)$ is independent of the behaviour of the function $g$ near $0$. In particular, without loss of generality, we can assume, in addition, that $g$ is a bounded function on $(0,\infty)$.

Here we define Dixmier traces in a form which is different from the original Dixmier definition in~\cite{D} (see e.g. \eqref{Dix_original_construction}). It is equivalent to the original one by~\cite[Theorem 17]{Sed_Suk}.

\begin{definition}\label{Dix_L} Let $g$ satisfy~\eqref{limcond} at $\infty$ and $g\notin L_1(0,\infty)$.
A symmetric functional $\phi$ on $\Lw(H)$ is called a Dixmier trace if 
$$\phi (X)= \frac1{\log 2}\omega \left( n \mapsto \frac{\sum_{k=0}^n \mu(k,X)}{G(n)} \right), \quad 0\le X\in \Lw(H), $$
for some extended limit $\omega$ on $\ell_\infty$. Here, $G$ is the primitive of $g$.
\end{definition}

\begin{remark}
In this paper we have chosen to normalise traces on $\Lw(H)$ by $\phi(\cD_g\chi_\Z)=1$. This allows to state \cref{one-to-one1} without a constant (compare to~\cite[Theorem 4.1]{SSUZ}). Usually Dixmier traces are normalised by $\phi({\rm diag}\{ g(n)\})=1$. To make Dixmier traces normalised in the same way as all other traces in this paper we must have a constant in the definition.
\end{remark}

%
%
%
%
%

\begin{remark}
The assumption $g\notin L_1(0,\infty)$ is crucial for the existence of Dixmier traces on $\mL_g(H)$. Indeed, if $g\in L_1(0,\infty)$, then the space $\mL_g(H)$ is a subspace of the trace class $\mL_1(H)$.
Moreover, the function $G(t)=\int_0^t g(s) ds$ is uniformly bounded on $[0,\infty)$ and $\lim_{t\to+\infty} G(t)$ exists. Using the properties of extended limits (see Proposition \ref{EL}), we obtain
\begin{align*}
	{\rm Tr}_\omega(X)&=\frac1{\log 2}\frac1{\lim_{t\to\infty}G(t)}\omega\left(n \mapsto \sum_{k=0}^n\mu(k,X)\right)\\
	&=\frac1{\log 2}\frac1{\lim_{t\to\infty} G(t)}\cdot \sum_{k=0}^\infty \mu(k,X),
\end{align*}
that is, in this case all Dixmier traces are scalar multiple of the standard trace.
Note, that Theorem \ref{one-to-one1} constructs all continuous symmetric functionals on $\mL_g(H)$ regardless of the integrability of $g$. 
\end{remark}

We now collect some of the properties of the function $g$.

\begin{lemma}\label{lem_g_prop}
Suppose that $g:(0,\infty)\to (0,\infty)$ is a bounded decreasing function, such that $g\notin L_1(0,\infty)$ and $\lim_{t\to\infty} \frac{g(t)}{g(2t)}=2$. Then, denoting by $G$ a primitive of $g$ we have that 
 \begin{equation}\label{RVcor}
			\frac{t g(t)}{G(t)}\to 0, \ t\to\infty,
			\end{equation}
and \begin{equation}\label{e0001}
				\frac{\sum_{m=0}^n 2^m g(2^m)}{G(2^n)} \to \log 2,
			\end{equation} as $n\to\infty$.

\end{lemma}
\begin{proof}
The assumption $\lim_{t\to\infty} \frac{g(t)}{g(2t)}=2$ means that $g$ is a regularly varying function of index $-1$. Therefore, \eqref{RVcor} follows from \cite[Theorem 1.5.11]{RegVar}. To prove \eqref{e0001} we firstly note that since  $g$ is monotone, we have that $\sum_{m=0}^{n} g(m)\sim G(n), \ n\to\infty.$
	Further, direct computation yields $\int_0^n 2^s g(2^s) ds = \log 2 \cdot G(2^n)$. 

Secondly, for every regularly varying $g$ of index $-1$, there exists a continuously differentiable function $\tilde{g}$, such that $\tilde{g}$ is a regularly varying function of index $-1$, $g \sim \tilde{g}$ and
	$\frac{t \tilde{g}'(t)}{\tilde{g}(t)} \to -1$ as $t\to\infty$ \cite[Theorem 1.8.2]{RegVar}. Thus, without loss of generality, we can assume that $g$ is continuously differentiable and $\frac{t g'(t)}{g(t)} \to -1$ as $t\to\infty$.	

Setting $f(s)=2^s g(2^s)$, we obtain
$$\frac{f'(s)}{f(s)}= \log 2 \left(1+\frac{2^s g'(2^s)}{g(2^s)}\right).$$
Since $\frac{t g'(t)}{g(t)} \to -1$ as $t\to\infty$, then $f'(s) = o(f(s))$ as $s\to\infty$. In particular, the Euler–Maclaurin formula reads:
\begin{align*}
	\int_0^n f(s) ds - \sum_{m=0}^n f(m) &= \frac{f(n)-f(0)}{2} + \int_0^n f'(s) B_1(s-\lfloor s \rfloor)ds\\
	&= \frac{f(n)-f(0)}{2} + o\left(\int_0^n f(s)ds\right),
\end{align*}
where $B_1$ is the Bernoulli polynomial and, so, is bounded on $[0,1]$. Therefore,
	$$\left|\int_0^n 2^s g(2^s) ds - \sum_{m=0}^n 2^m g(2^m) \right| \le \max\{g(1), 2^n g(2^n)\}+ o\left(\int_0^n 2^s g(2^s)ds\right).$$
		
	Dividing this by $G(2^n)$ we obtain
	\begin{equation*}\label{eqG}
			\left|\log 2 - \frac{\sum_{m=0}^n 2^m g(2^m)}{G(2^n)} \right| \le \frac{\max\{g(1), 2^n g(2^n)\}}{G(2^n)} +o(1).
		\end{equation*}
Using ~\eqref{RVcor} and that $g\notin L_1(0,\infty)$, we conclude that
		\begin{equation*}
				\frac{\sum_{m=0}^n 2^m g(2^m)}{G(2^n)} \to \log 2,
			\end{equation*} as $n\to\infty$.
\end{proof}

To characterise Dixmier traces on $\Lw(H)$ we introduce weighted Ces\`aro (Riesz) operators $C_g : \ell_\infty \to \ell_\infty$ in the following way (see e.g.~\cite[Section 3.2]{Boos}):
\begin{equation}\label{riesz}
	(C_gx)_n=\frac{\sum_{m=0}^n 2^m g(2^m) x_m}{\sum_{m=0}^n 2^m g(2^m)}, \ x \in \ell_\infty, \ n\ge0.
\end{equation}

In the special case when $g(t)=1/t$ the operator $C_g$ is the classical Ces\`aro operator.
Clearly, $C_g \chi_{\Z_+} = \chi_{\Z_+}$.
Since $g$ is a positive function, it is easy to see that $C_gx\ge0$ if $x\ge0$. However, in general the method of summation defined by $C_g$ is not regular. 
Recall that a summation method is regular if it maps convergent sequences to convergent preserving the limit value.

\begin{proposition}\label{reg}
	If $g : [0,\infty) \to (0,\infty)$ is a bounded  decreasing function such that $g\notin L_1(0,\infty)$, then 
 $C_g$ defines a regular summation method. If additionally, $g$ satisfies~\eqref{limcond} at $\infty$, then the operator $C_g(I-S_+)$ maps $\ell_\infty$ into $c_0 := c_0(\Z_+)$.
\end{proposition}

\begin{proof}
	By \cref{lem_l1}, we have $\sum_{m=0}^n 2^m g(2^m) \to \infty, n\to \infty$. The first assertion follows from~\cite[Theorem 3.2.7]{Boos}.
	
%
%
%
%
	
	Let $x\in \ell_\infty$ and $n\in \N$. By definition of $C_g$ we have
	\begin{align*}
		&(C_g(I-S_+)x)_n =\frac{g(1) x_0+\sum_{m=1}^n 2^m g(2^m) (x_m-x_{m-1})}{\sum_{m=0}^n 2^m g(2^m)}\\
		&=\frac{\sum_{m=0}^{n} x_m (2^m g(2^m)-2^{m+1} g(2^{m+1}))+2^{n+1} g(2^{n+1}) x_{n+1}}{\sum_{m=0}^n 2^m g(2^m)}\\
		&=\frac{\sum_{m=0}^{n} 2^m g(2^m) x_m (1-\frac{2 g(2^{m+1})}{g(2^m)})+2^{n+1} g(2^{n+1}) x_{n+1}}{\sum_{m=0}^n 2^m g(2^m)}\\
		&= (C_g y)_n + \frac{2^{n+1} g(2^{n+1}) x_{n+1}}{\sum_{m=0}^n 2^m g(2^m)},
	\end{align*}
	where $y:= \left\{ x_m (1-\frac{2 g(2^{m+1})}{g(2^m)})\right\}_{m=0}^\infty$. Since $\frac{g(t}{g(2t)}\to 2$ as $t\to\infty$, it follows that $y\in c_0$.  Since $C_g$ is regular, we infer that the first summand vanishes as $n\to\infty$.
	
Combining  \eqref{RVcor} with \cref{lem_l1}, we infer that 
	\begin{equation}\label{e001}
			\frac{2^{n} g(2^{n})}{\sum_{m=0}^n 2^m g(2^m)} \to 0,
		\end{equation} as $n\to\infty$. Hence, $C_g(I-S_+)x \in c_0,$ as required.
\end{proof}

\begin{remark}
	In general, the condition $g\notin L_1$ in the above proposition is essential. For instance, the function  $$g(t)=\frac1{(t+2)\log^2(t+2)}, \ t\ge0,$$ satisfies all assumptions of \cref{reg}, but $g\in L_1$. In this case, $2^m g(2^m)=O(m^{-2})$, and therefore, $\sum_{m=0}^n 2^m g(2^m) \nrightarrow \infty, n\to \infty$ and, thus, $C_g$ is not regular.
	
	For the same $g$ and the sequence $x=(1,0,0, \dots)$ a direct verification shows that the sequence $C_g(I-S_+)x$ does not belong to $c_0$.

	The assumption of regular variation (that is, $g$ satisfies \eqref{limcond}) is essential, too. Consider
	$$g = \chi_{(0,1)} + \sum_{i=0}^\infty \frac{2^{-2i}}i \chi_{[2^{2i}, 2^{2i+2})}.$$
	We have $g(2^{2i}) = g(2^{2i+1}) = 2^{-2i} / i$ and $2^m g(2^m) = 1/m$ or $2/m$ for even and odd $m$, respectively. Then,
	$$(C_g(I-S_+)x)_n =\frac{\sum_{m=0}^{n} x_m (-1)^m 1/m }{3 \sum_{m=0}^n 1/m} + o(1),$$
	does not vanish for $x = \{(-1)^m\}_{m\ge0}$.
\end{remark}

The following result is an extension of~\cite[Theorem 5.8]{SSUZ}, where it was proved for Dixmier traces on $\mathcal L_{1,\infty}(H)$.

\begin{theorem}\label{DixChar}
	Let $g : [0,\infty) \to (0,\infty)$ be a bounded decreasing function satisfying~\eqref{limcond}, such that $g\notin L_1(0,\infty)$. A symmetric functional $\phi$ on $\Lw(H)$ is a Dixmier trace if and only it can be written in the form
	\begin{equation}\label{eq001}
		\phi(X) = (\gamma \circ C_g)(\Phi_g(X)), \ 0\le X \in \Lw(H),
	\end{equation}
	for some extended limit $\gamma$ on $\ell_\infty$.
\end{theorem}

\begin{proof}
We firstly show that for any extended limit $\gamma$ on $\ell_\infty$, the formula \eqref{eq001} defines a Dixmier trace on  $\Lw(H)$.
For every extended limit $\gamma$ the functional $\gamma \circ C_g$ is continuous and linear on $\ell_\infty$. The $S_+$-invariance of this functional follows from \cref{reg}.
	By \cref{one-to-one1} the formula
	$$\phi(X)=(\gamma \circ C_g)\left(\{\frac{\tau(X_k)}{2^{k} g(2^k)}\}_{k\in\Z_+}\right), \quad X\in \mL_g$$
	defines a symmetric functional on $\mL_g(H)$ for every
	dyadic $\ell_g$-representation $X=\sum_{k\in\Z_+} X_k$. By \cite[Lemma 2.7]{LU} every positive operator $X\in \mL_g(H)$ admits a dyadic $\ell_g$-representation with 
	$X_k = X E^X_{\Delta_k}$, where $\Delta_k = (\mu(2^k, X), \mu(2^{k+1}, X)]$. 
	By \cite[Proposition 3.3.9 (ii)]{DPS} we have that 
	$\tau(X_k) = \int_{2^k}^{2^{k+1}} \mu(s,X) ds$
and, so definition of $\Phi_g$ (see \eqref{def_Phi_g}) implies that
	\begin{equation}\label{eq500}
		\phi(X)=(\gamma \circ C_g) (\Phi_g X), 0\le X\in\mL_g(H).
	\end{equation}
	 We shall show that this functional coincides with a Dixmier trace.
	
	Using the definition of the operators $C_g$ and $\Phi_g$, we obtain
	\begin{equation}\label{eq000}
		\phi(X) = \gamma\left(  n \mapsto \frac1{\sum_{m=0}^n 2^m g(2^m)} \int_{0}^{2^{n+1}} \mu(s,X) ds\right).
	\end{equation}
	
Referring to \eqref{e0001} and using the fact that $\gamma$ is extended limits, we obtain that 
	$$ \phi(X) = \frac1{\log 2} \gamma\left( n \mapsto  \frac1{G(2^n)} \int_{0}^{2^{n+1}} \mu(s,X) ds \right). $$
	
	By \cref{simple} and \eqref{RVcor} we obtain
	\begin{equation}\label{e002}
		\int_{2^{n}}^{2^{n+1}} \mu(s,X) ds = O(2^n g(2^n))=o(G(2^n)), \ n\to\infty.
	\end{equation} 
	
	Since every extended limit vanishes on $c_0$, it follows that
	\begin{align*}
		\phi(X) &= \frac1{\log 2} \gamma\left( n \mapsto  \frac1{G(2^n)} \int_{0}^{2^{n}} \mu(s,X) ds \right)\\
		&= \frac1{\log 2} \gamma_1 \left(  n \mapsto \frac1{G(n)} \int_{0}^{{n}} \mu(s,X) ds\right),
	\end{align*}
	where $\gamma_1(x):=\gamma( k \mapsto x_{2^k})$ is an extended limit. It follows from \cref{Dix_L} that $\phi$ is a Dixmier trace on $\Lw(H)$.
	
Conversely, let $\phi$ be a Dixmier trace on $\Lw(H)$. Thus, for some extended limit $\omega$ on $\ell_\infty$ we have
	$$\phi(X)=\frac1{\log 2} \gamma \left( n \mapsto \frac1{G(n)} \int_{0}^{{n}} \mu(s,X) ds \right), \quad 0\le X\in \Lw(H). $$
	
	By \cref{one-to-one1} the trace $\phi$ corresponds to the $S_+$-invariant functional $\theta$ on $\ell_\infty$ of the form
	$$\theta(x)=\phi(\cD_gx)=\frac1{\log 2} \gamma \left( n \mapsto \frac1{G(n)} \int_{0}^{{n}} \mu(s,\cD_gx) ds \right), \ x \in \ell_\infty.$$
	By \cite[Theorem 4.5]{LU} $\mu(\cD_gx) = \mu(D_gx)$ for every $x \in S(\Z_+)$. Further, it follows from \cite[Lemma 2.2]{FigielKalton}
	$$\left|\int_{0}^{t} \mu(s, D_gx) ds - \int_{0}^{t} (D_gx)(s) ds\right| \le 8 c\cdot  t g(t),$$
	for every $t>0$ and some constant $c>0$. By \eqref{RVcor} we have that 
	$$\frac1{G(n)}\left|\int_{0}^{n} \mu(s, D_gx) ds - \int_{0}^{n} (D_gx)(s)\right| \to 0, \ n \to \infty.$$

	Since $\gamma$ is an extended limit (at $+\infty$), we obtain
	\begin{align*}
		\theta(x)&=\frac1{\log 2} \gamma \left( n \mapsto \frac1{G(n)} \int_{0}^{{n}} (D_gx)(s) ds\right)\\
		&=\frac1{\log 2} \gamma_1 \left( n \mapsto \frac1{G(2^n)} \int_{0}^{{2^n}} (D_gx)(s) ds\right)
	\end{align*}
	where $\gamma_1(x):=\gamma(x_{\lfloor \log_2n \rfloor})$ is an extended limit. 
	
	Using the definition of operator $D_g$ and \eqref{e0001}, we obtain
	\begin{align*}
		\theta(x)&=\gamma_1 \left( n \mapsto \frac{\sum_{m=0}^n 2^m g(2^m) x_m}{\sum_{m=0}^n 2^m g(2^m)} \right),  \ x \in \ell_\infty.
	\end{align*}
%
\end{proof}

The following theorem shows that the mapping $i_{f,g}$ given by~\eqref{i} is not a bijection (in general) between the set of Dixmier traces on $\mathcal L_f(H)$ and that on $\Lw(H)$.

\begin{theorem}\label{Dix_cor}
	Let 
	$$f(t)= \begin{cases} 1, 0<t<1,\\ 1/t, t\ge1\end{cases} \text{and}  \
	g(t)= \begin{cases} 1/2, 0<t<2,\\ \frac1{t\log_2 t}, t\ge2\end{cases}.$$  
	The mapping $i_{g,f}$ given by~\eqref{i} maps the set of all Dixmier traces on $\mL_g(H)$ into the set of  all Dixmier traces on $\mL_f(H)$, but it is not surjective.
\end{theorem}

\begin{remark}
	Note that for such function $f$ the ideal $\mathcal{L}_f(H)$ coincide with the weak trace-class ideal $\mathcal{L}_{1,\infty}(H)$ and the operator $C_f$ is nothing but the usual Ces\`aro operator.
\end{remark}

\begin{proof}
	Note, that both functions $f$ and $g$ satisfy the assertions of \cref{DixChar} and, so both spaces $\mL_f(H)$ and $\mL_g(H)$ admit Dixmier traces.
	
	First, we show that for every Dixmier trace $\phi$ on $\Lw(H)$ the functional $\phi_1:= i_{g,f}(\phi)$ is a Dixmier trace on $\mathcal{L}_f(H)$. Using~\eqref{i}, \cref{DixChar} and \cref{inv1}(ii), we obtain 
	\begin{equation}\label{eqtr}
		\phi_1=\phi \circ \cD_g \circ \Phi_f= \omega \circ C_g \circ \Phi_g \circ \cD_g \circ \Phi_f= \omega \circ C_g  \circ \Phi_f,
	\end{equation}
	where $\omega$ is an extended limit on $\ell_\infty$.
	
	We shall show that $\phi_1=\gamma \circ C  \circ \Phi_f$ for some extended limit $\gamma$ on $\ell_\infty$. 
	
	The definition of the Riesz means yields
	\begin{equation}\label{Cesaro_g}
		(C_gx)_n = \frac{\frac{x_0}2+ \sum_{k=1}^n \frac{x_k}k}{\sum_{k=1}^n \frac1k+1/2}, \ x\in \ell_\infty, \ n\ge0.
	\end{equation}

Setting
$$(Mx)_0 = x_0, \quad (Mx)_n = \frac{\sum_{k=1}^n \frac{x_k}k}{\log (n+1)}, \ n\ge1,  \ x \in \ell_\infty,$$	
we obtain  $M-C_g : \ell_\infty \to c_0$. 

	Direct calculations show that the inverse of the operator $C$ is given by 
	$$(C^{-1}x)_0=x_0, \ \text{and} \ (C^{-1}x)_n=(n+1)x_n-nx_{n-1}, \ n\ge1.$$
	
	Thus, for every $x\in \ell_\infty$ and $n\ge1$ we obtain
	\begin{equation}\label{CC}
		\begin{aligned}
			(M C^{-1}x)_n &= \frac{\sum_{k=1}^n \frac1k ((k+1)x_k-kx_{k-1})}{\log (n+1)}\\
			&= \frac{\sum_{k=1}^n (x_k-x_{k-1}+\frac{x_k}{k} )}{\log (n+1)}\\
			&= \frac{x_n+ \sum_{k=1}^n \frac{x_k}{k}}{\log (n+1)}\\
			&= (Mx)_n+\frac{x_n}{\log (n+1)}.
		\end{aligned}
	\end{equation}
%
%
%
Since $\omega$ is an extended limit, it follows that
	$$\phi_1= \omega \circ M \circ \Phi_f= \omega \circ M \circ C^{-1}\circ C\circ \Phi_f = \gamma\circ C\circ \Phi_f, $$
	for an extended limits $\gamma: = \omega \circ M \circ C^{-1}= \omega \circ M = \omega \circ C_g$ on $\ell_\infty.$
	Thus, $\phi_1$ is a Dixmier trace on $\mathcal L_f(H)$ by \cref{DixChar}.
	
	Next, we show that the mapping $i_{g,f}$ is not surjective. Let $\mathfrak D_1$ be the set of all Dixmier traces on $\mathcal L_f(H)$ and let 
	$\mathfrak D_2$ be the set of all traces on $\mathcal L_f(H)$ of the form~\eqref{eqtr}. 
	
	For every $\phi_2\in \mathfrak D_2$ we have that 
	$\phi_2(\cD_fx) = (\omega_2 \circ M)(\Phi_f \cD_fx)$ for some extended limit $\omega_2$ and every $x\in \ell_\infty$.
	It follows from \cref{inv1} (ii), ~\cref{reg} and the fact that extended limits vanish on $c_0$, that
	$\phi_2(\mathcal D_fx) = \omega_2 (M x)$.
	Similarly, for every $\phi\in \mathfrak D_1$ we have 
	\begin{equation}\label{eq300}
		\phi_1(\cD_fx) = (\omega_1 \circ C)(\Phi_f \cD_fx)= \omega_1 (C x)
	\end{equation}
	for some extended limit $\omega_1$.

	By~\cite[Lemma 32]{SUZ1} $M\circ C - M : \ell_\infty \to c_0$. This implies that for every extended limit $\omega$ on $\ell_\infty$ the functional $\omega\circ M$ is Ces\`aro invariant Banach limit. It follows from~\cite[Theorem 20]{SSU3} that every Ces\`aro invariant Banach limit equals $1/2$ on the sequence
	$$y=\sum_{n=1}^\infty \chi_{[2^{2n}, 2^{2n+1})}.$$ Hence,
	$\phi_2(D_fy) =\frac12$ for every $\phi_2\in \mathfrak D_2$.

	On the other hand, by \eqref{eq300} and \cref{EL} we obtain
	\begin{align*}
		\{\phi(\cD_fy) : \phi\in \mathfrak D_1 \} &= \{ \omega (C y) : \omega \ \text{is extended limit}\}\\
		&=[\liminf_{n\to \infty} (Cy)_n,\limsup_{n\to \infty} (Cy)_n].
	\end{align*}	
	Direct calculations show that
	$$\limsup_{n\to \infty} (Cy)_n=\limsup_{n\to \infty} (Cy)_{2^{2n+1}-1}=\frac23$$
	and
	$$\liminf_{n\to \infty} (Cy)_n=\liminf_{n\to \infty} (Cy)_{2^{2n}-1}=\frac13.$$
	
	Thus,
	\begin{align*}
		\{\phi(\cD_fy) : \phi \in \mathfrak D_1 \} = [\frac13, \frac23],
	\end{align*}
that is, there are traces from $\mathfrak D_1$ which do not equal to $1/2$ on $D_fy$. 
	This proves that $\mathfrak D_1\neq \mathfrak D_2$.
\end{proof}

\section{EXTENSION OF CONNES' TRACE FORMULA}\label{sec:app}

In this section we prove a variant of the Connes' trace theorem for continuous traces and a wide class of integral operators. 
In particular, this class contains log-polyhomogeneous pseudo-differential operators studied in \cite{Lesch_log_PDO}.

We denote by $C_c^\infty(\R^d)$ the space of smooth compactly supported functions. Recall that an operator $X:L_2(\R^d)\to L_2(\R^d)$ is said to be compactly supported if there exist $f,g\in C_c^\infty(\mathbb R^d)$ such that $M_fXM_g=X$. Here, $M_f$ and $M_g$ stand for the multiplication operators.


Let $\Delta=\sum_{j=1}^d\frac{\partial^2}{\partial x_j^2}$ be the Laplace operator on $\R^d$. For $q>0$ we denote by $\Lw^{q}(L_2(\R^d))$ the $q$-convexification on $\Lw(L_2(\R^d))$, that is $\Lw^{q}(L_2(\R^d))$ consists of all operators $X$ such that $|X|^q \in \Lw(L_2(\R^d))$. Also we will use the standard notation $\langle x \rangle := (1+x^2)^{1/2}$.

The following definitions were introduced in~\cite{GU}.

\begin{definition}
	\label{weaklymod}
	Let $g:[0,\infty)\to(0,\infty)$ satisfy~\eqref{limcond} and  be a decreasing function.
	
	(i) Let $V \in B(L_2(\R^d))$ be strictly positive.  We say that $X\in {B}(L_2(\R^d))$ is weakly $g$-modulated with respect to $V$ if there is $p\ge 1$ such that the densely defined operator $XV^{-1/p}$ extends to a bounded 
	the operator with $XV^{-1/p}\in \Lw^{(\frac{p}{p-1})}(L_2(\R^d))$. 
	
	(ii) We say that $X\in {B}(L_2(\R^d))$ is weakly $g$-Laplacian modulated if it is weakly $g$-modulated with respect to $V=g((1-\Delta)^{d/2})$. 
	
\end{definition}

\begin{remark}
	The operator $g((1-\Delta)^{d/2})$ is not in $\Lw(L_2(\R^d))$, it is not even a compact operator on $L_2(\R^d)$.
	However, it is locally in $\Lw(L_2(\R^d))$ in the sense that $f \cdot g((1-\Delta)^{d/2})$, $g((1-\Delta)^{d/2}) \cdot f \in \Lw(L_2(\R^d))$
	for every $f\in C_c^\infty(\R^d)$.
\end{remark}

We will introduce a stronger version of $g$-reasonable decay (see \cite[Definition 5.8]{GU}). Here, by $p_X$ we denote the $L_2$-symbol  of a Hilbert-Schmidt operator $X$ (see e.g. \cite[Theorem 1.5.4]{LMSZ_book}).
\begin{definition}
	\label{stphidec}
	Let $g:[0,\infty)\to(0,\infty)$ satisfy~\eqref{limcond} and   be a decreasing function. We say that $X\in \mathcal{L}_2(L_2(\R^d))$ has sharp $g$-reasonable decay if 
	$$\int_{\R^{2d}} \frac{|p_X(x,\xi)|}{\langle t-\langle\xi\rangle^d\rangle}\, dx d\xi =O(t g(t)), \quad\mbox{as $t\to \infty$}.$$
\end{definition}

\begin{remark}
	Since the function $\xi \mapsto \langle t-\langle\xi\rangle^d\rangle^{-1}$ belongs to $L_2(\R^d)$, if follows that $(x,\xi)\mapsto \frac{|p_X(x,\xi)|}{\langle t-\langle\xi\rangle^d\rangle}$ is integrable for every $L_2$-function $p_X$.
\end{remark}

The following result is Lemma 3.5 in~\cite{GU}.

\begin{lemma}
	\label{l_exp}
	Let $g:[0,\infty)\to(0,\infty)$ be a decreasing function satisfying~\eqref{limcond}. Let $V \in \Lw(L_2(\R^d))$ be strictly positive and 
	let $\{e_n\}_{n\in \N}$ be an eigenbasis for $V$ ordered so that $Ve_n = \mu(n,V)e_n$, $n\ge0$. If $X\in  {B}(L_2(\R^d))$ is weakly $g$-modulated with respect to $V$, then
	$X \in \Lw(L_2(\R^d))$ and
	\begin{equation}
		\label{lemma11210}
		\sum_{k=0}^n \lambda(k,\Re X) - \sum_{k=0}^n \langle(\Re X)e_k, e_k\rangle = O(n g(n)), \ n \to \infty.
	\end{equation}
	Here $\Re X= \frac{X^*+X}2$ denotes the real part of $X$.
\end{lemma}

\begin{remark}
	The remainder in the statement of Lemma 3.5 in the paper~\cite{GU} is $o\left(\int_0^n g(s)ds\right).$ However, as one can see from the proof of \cite[Lemma 3.4]{GU}, it is possible to sharpen this remainder to be $O(n g(n))$. 
\end{remark}

Let $\mathbb{T}^d$ be the $d$-torus equipped with its flat metric and let $\Delta_{\mathbb{T}^d}$ be the associated Laplacian. 
Denote by $\mathrm{e}_\mathbf{k}(u):=\mathrm{e}^{2\pi i\langle \mathbf{k},u\rangle}$, $\mathbf{k}\in \Z^d$  the eigenbasis basis for $\Delta_{\mathbb{T}^d}$.

\begin{lemma}\label{comp_supp}
	Let $g:[0,\infty)\to(0,\infty)$ be a decreasing function satisfying~\eqref{limcond} and let $X \in  {B}(L_2(\R^d))$ be compactly supported in $(0,1)^d$ operator with strict $g$-reasonable decay. One has
	$$\sum_{\langle \mathbf{k}\rangle^d\leq t}\langle X\mathrm{e}_\mathbf{k},\mathrm{e}_\mathbf{k}\rangle_{L_2((0,1)^d)}-\int_{\langle\xi\rangle^d\leq t}\int_{\R^d} p_X(x,\xi) dx d\xi = O(t g(t)), \quad\mbox{as $t\to \infty$}.$$
\end{lemma}

\begin{proof}
In the way identical to~\cite[Lemma 11.4.6]{LSZ} it can be shown that

$$\left|\sum_{\langle \mathbf{k}\rangle^d\leq t}\langle X\mathrm{e}_\mathbf{k},\mathrm{e}_\mathbf{k}\rangle_{L_2((0,1)^d)}-\int_{\langle\xi\rangle^d\leq t}\int_{\R^d} p_X(x,\xi) dx d\xi\right| \le C \cdot \int_{\R^{2d}} \frac{|p_X(x,\xi)|}{\langle t-\langle\xi\rangle^d\rangle}\, dx d\xi,$$
for some constant $C>0$.
The assertion follows from the definition of strict $g$-reasonable decay.
\end{proof}

The following result is an extension of Connes' trace formula for continuous traces on principal ideals $\Lw(L_2(\R^d))$.

\begin{theorem}
	Let $g$ satisfy~\eqref{limcond} and let $X$ be a compactly supported, weakly $g$-Laplacian modulated operator with strict $g$-reasonable decay. Then $X\in \Lw(L_2(\R^d))$ and 
	
	(i) for every continuous trace $\phi$ on $\Lw(L_2(\R^d))$ we have
	$${\phi}(X) = \theta \left(n\mapsto \frac1{2^n g(2^n)} \int_{\R^d} \int_{2^n \le \langle \xi\rangle^d\le 2^{n+1}} p_X(x,\xi)\, d\xi dx\right),$$
	where $\theta$ is a continuous $S_+$-invariant functional on $\ell_\infty$;
	
	(ii) All continuous normalised singular traces on $\Lw(L_2(\R^d))$ take the same value on the operator $X$ if and only if a sequence
	$$\left\{\frac1{2^n g(2^n)} \int_{\R^d} \int_{2^n \le \langle \xi\rangle^d\le 2^{n+1}} p_X(x,\xi)\, d\xi  dx\right\}_{n\ge0}$$
	is almost convergent (that is all Banach limits take the same value).
	
\end{theorem}

\begin{proof} The fact that $X\in \Lw(L_2(\R^d))$ was established in~\cite[Lemma 3.4]{GU}.
	
	(i) Without loss of generality, we can assume that $X$ is compactly supported in $(0,1)^d$. Thus, \cite[Lemma 6]{GU} implies that $X$ is weakly $g$-modulated with respect to $g((1-\Delta_{\mathbb{T}^d})^{d/2})$.
	Combining \cref{l_exp} and~\cref{comp_supp}, we obtain
	$$\sum_{k=0}^n [\lambda(k,\Re X)+i \lambda(k,\Im X)]-\int_{\langle\xi\rangle^d\leq n}\int_{\R^d} p_X(x,\xi) dx d\xi = O(n g(n)), t\to \infty. $$
	
	Thus,
	$$\sum_{k=0}^{2^{n+1}-2} \left[\lambda(k,\Re X)+i \lambda(k,\Im X)-\int_{ k-1\leq \langle\xi\rangle^d\leq k}\int_{\R^d} p_X(x,\xi) dx d\xi\right]$$
	$$ = O((2^{n+1}-2)g(2^{n+1}-2))= O(2^{n}g(2^{n})), n\to \infty, $$
	since $g$ satisfies condition~\eqref{limcond}.
	
	It follows from \cref{aconv1} that the sequence
	$$\left\{ \frac{1}{2^n g(2^n)}\sum_{k=2^n-1}^{2^{n+1}-2} \left[\lambda(k,\Re X)+i \lambda(k,\Im X)-\int_{ k-1\leq \langle\xi\rangle^d\leq k}\int_{\R^d} p_A(x,\xi) dx d\xi\right] \right\}_{n \ge 0}$$
	belongs to $\Ra +c_0.$
	Equivalently, the sequence
	$$
	\frac{1}{2^n g(2^n)}\sum_{k=2^n-1}^{2^{n+1}-2} [\lambda(k,\Re X)+i \lambda(k,\Im X)]-\frac{1}{2^n g(2^n)}\int_{ 2^n\leq \langle\xi\rangle^d\leq 2^{n+1}}\int_{\R^d} p_A(x,\xi) dx d\xi,
	$$
	$n \ge 0 $,
	belongs to $\Ra +c_0.$ By the linearity of traces and \cref{one-to-one1}, for every continuous trace $\phi$ on $\Lw(L_2(\R^d))$ we have 
	$$\phi(X) = \phi(\Re X)+i \phi(\Im X)= \theta \left( n\mapsto \frac{1}{2^n g(2^n)}\sum_{k=2^n-1}^{2^{n+1}-2} [\lambda(k,\Re X)+i \lambda(k,\Im X)]\right),$$
	where $\theta$ is a continuous $S_+$-invariant functional on $\ell_\infty$.
	Since every continuous $S_+$-invariant functional on $\ell_\infty$ vanishes on $\Ra +c_0$, it follows that
	\begin{equation*}
		\phi(X) =\theta \left(n\mapsto \frac{1}{2^n g(2^n)}\int_{ 2^n\leq \langle\xi\rangle^d\leq 2^{n+1}}\int_{\R^d} p_X(x,\xi) dx  d\xi\right),
	\end{equation*}
	as required.
	
	(ii) By \cref{NCT} all continuous normalised singular traces on $\Lw(L_2(\R^d))$ take the same value on the operator $X$ if and only if all positive normalised singular traces on $\Lw(L_2(\R^d))$ coincide on $X$. By part (i), this is the case if and only if there exists $c\in \R^d$ such that
	$$ \theta \left(n\mapsto \frac{1}{2^n g(2^n)}\int_{ 2^n\leq \langle\xi\rangle^d\leq 2^{n+1}}\int_{\R^d} p_X(x,\xi) dx d\xi\right)=c$$
	for every positive normalised $S_+$-invariant functional $\theta$ on $\ell_\infty$, that is for every Banach limit. Since Banach limits take the same value on almost convergent sequences only~\cite{L}, it follows that
	the sequence
	$$\left\{\frac1{2^n g(2^n)}  \int_{2^n \le \langle \xi\rangle^d\le 2^{n+1}} \int_{\R^d}p_X(x,\xi)\,  dx d\xi\right\}_{n\ge0}$$
	is almost convergent.
\end{proof}



\end{document}